\documentclass{amsart}
\usepackage{amsmath,amssymb, amsthm}
\usepackage{graphicx} 
\usepackage{color} 
\newtheorem{definition}{Definition}[section]
\newtheorem{proposition}[definition]{Proposition}
\newtheorem{lemma}[definition]{Lemma}
\newtheorem{theorem}[definition]{Theorem}
\newtheorem{corollary}[definition]{Corollary}
\newtheorem{remark}[definition]{Remark}
\title{Channel surfaces in Lie sphere geometry}
\author{Mason Pember and Gudrun Szewieczek}
\address[M.~Pember]{Vienna University of Technology,
Wiedner Hauptstra\ss e 8-10/104, A-1040 Vienna. Austria.}
\email{mason@geometrie.tuwien.ac.at}

\address[G.~Szewieczek]{Vienna University of Technology,
Wiedner Hauptstra\ss e 8-10/104, A-1040 Vienna. Austria.}
\email{gudrun@geometrie.tuwien.ac.at}

\begin{document}

\maketitle
\begin{abstract}
We discuss channel surfaces in the context of Lie sphere geometry and characterise them as certain $\Omega_{0}$-surfaces. Since $\Omega_{0}$-surfaces possess a rich transformation theory, we study the behaviour of channel surfaces under these transformations. Furthermore, by using certain Dupin cyclide congruences, we characterise Ribaucour pairs of channel surfaces. 
\end{abstract}

\section{Introduction}
Channel surfaces, that is, envelopes of 1-parameter families of spheres, have been intensively studied for many years. Although these surfaces are a classical notion (e.g., \cite{B1929,L1872,M1850}), they are also a subject of interest in recent research. For example, in~\cite{H2003} channel surfaces were studied in the context of M\"{o}bius geometry, channel linear Weingarten surfaces were characterised in \cite{HMT2015} and the existence of a rational parametrisation was investigated in \cite{PP1997}.  Moreover, channel surfaces are widely used in Computer Aided Geometric Design.

In this paper, following the example of~\cite{B1929}, we study these surfaces in Lie sphere geometry using the hexaspherical coordinate model introduced by Lie \cite{L1872}. Applying the gauge theoretic approach of \cite{BHPR2017,C2012i,P2017}, we show that Legendre immersions parametrising channel surfaces are the $\Omega_0$-surfaces that admit a linear conserved quantity. This approach lends itself well to studying the transformation theory that $\Omega_{0}$-surfaces possess. We give special attention to the Lie-Darboux transformation of $\Omega_{0}$-surfaces, a particular Ribaucour transformation. We show that any Lie-Darboux transform of a channel surface is again a channel surface. Furthermore, after choosing the appropriate $\Omega_{0}$-structures, any Ribaucour pair of channel surfaces (with corresponding circular curvature lines) is a Lie-Darboux pair.     

We characterise Ribaucour pairs of umbilic-free Legendre immersions in terms of a special pair of Dupin cyclide congruences enveloping both surfaces. In Section~\ref{sec:rib} we investigate the behaviour of these Dupin cyclide congruences when both Legendre immersions participating in this Ribaucour pair parametrise channel surfaces. In a similar vein, given a pair of sphere curves, we construct two 1-parameter families of Dupin cyclides whose coincidence determines when the envelopes of the sphere curves form a Ribaucour pair (with corresponding circular curvature lines). 

In Section~\ref{sec:sym} we apply our theory of channel surfaces to the special case of curves in conformal geometry. We recover a result of~\cite{BH2006}, showing how the classical notion of Ribaucour transforms of curves is related to the Ribaucour transforms of Legendre immersions parametrising these curves. 

\textit{Acknowledgements.} The authors would like to give special thanks to Professor Udo Hertrich-Jeromin, who encouraged them to embark on this project and provided many insightful comments. They would also like to thank Professor Francis Burstall for his valuable feedback regarding this paper. This research project began whilst the first author was an International Research Fellow of the Japan Society for the Promotion of Science (JSPS) and continued with the support of the Austrian Science Fund (FWF) through the research project P28427-N35 ``Non-rigidity and symmetry breaking". The second author is grateful for financial support from the grant of the FWF/JSPS-Joint project I1671-N26 ``Transformations and Singularities'', which gave her the possibility to visit Kobe University, where the crucial ideas of this paper were developed.

\section{Preliminaries} 

Given a vector space $V$ and a manifold $\Sigma$, we shall denote by $\underline{V}$ the trivial bundle $\Sigma\times V$. Given a vector subbundle $W$ of $\underline{V}$, we define the derived bundle of $W$, denoted $W^{(1)}$, to be the subset of $\underline{V}$ consisting of the images of sections of $W$ and derivatives of sections of $W$ with respect to the trivial connection on $\underline{V}$. In this paper, most of the derived bundles that appear will be vector subbundles of the trivial bundle, but in general this is not always the case as, for example, the rank of the derived bundle may not be constant over $\Sigma$. 

Throughout this paper we shall be considering the pseudo-Euclidean space $\mathbb{R}^{4,2}$, i.e., a 6-dimensional vector space equipped with a non-degenerate symmetric bilinear form $(\,,\,)$ of signature $(4,2)$. Let $\mathcal{L}$ denote the lightcone of $\mathbb{R}^{4,2}$. According to Lie's~\cite{L1872} correspondence, points in the projective lightcone $\mathbb{P}(\mathcal{L})$ correspond to spheres in any three dimensional space form. A detailed modern account of this is given in~\cite{C2008}. Given a manifold $\Sigma$ we then have that any smooth map $s:\Sigma\to \mathbb{P}(\mathcal{L})$ corresponds to a sphere congruence in any space form. We shall thus refer to $s$ as a~\textit{sphere congruence}. Such a map can also be identified as a smooth rank 1 null subbundle of the trivial bundle $\underline{\mathbb{R}}^{4,2}$. 

The orthogonal group $\textrm{O}(4,2)$ acts transitively on $\mathcal{L}$ and thus acts transitively on $\mathbb{P}(\mathcal{L})$. In~\cite{C2008} it is shown that $\textrm{O}(4,2)$ is a double cover for the group of Lie sphere transformations. The Lie algebra $\mathfrak{o}(4,2)$ of $\textrm{O}(4,2)$ is well known to be isomorphic to the exterior algebra $\wedge^{2}\mathbb{R}^{4,2}$ via the identification 
\[ a\wedge b\, (c) = (a,c)b - (b,c)a,\]
for $a,b,c\in\mathbb{R}^{4,2}$. We shall frequently use this fact throughout this paper. 

Given a manifold $\Sigma$, we define the following product of two vector-valued 1-forms $\omega_{1},\omega_{2}\in\Omega^{1}(\underline{\mathbb{R}}^{4,2})$:
\[ \omega_{1}\curlywedge \omega_{2}(X,Y) := \omega_{1}(X)\wedge \omega_{2}(Y) - \omega_{1}(Y)\wedge \omega_{2}(X),\]
for $X,Y\in\Gamma T\Sigma$. Hence, $\omega_{1}\curlywedge \omega_{2}$ is a $2$-form taking values in $\wedge^{2}\underline{\mathbb{R}}^{4,2}$. Notice that $\omega_{1}\curlywedge \omega_{2} = \omega_{2}\curlywedge \omega_{1}$. 

Recall that we also have the following product for two $\mathfrak{so}(4,2)$-valued 1-forms $A,B\in \Omega^{1}(\mathfrak{so}(4,2))$:
\[ [A\wedge B](X,Y) = [A(X),B(Y)]-[A(Y),B(X)],\]
for $X,Y\in \Gamma T\Sigma$.

\subsection{Legendre maps}
Let $\mathcal{Z}$ denote the Grassmannian of isotropic 2-dimensional subspaces of $\mathbb{R}^{4,2}$.
Suppose that $\Sigma$ is a 2-dimensional manifold and let $f:\Sigma\to \mathcal{Z}$ be a smooth map. By viewing $f$ as a 2-dimensional subbundle of the trivial bundle $\underline{\mathbb{R}}^{4,2}$, we may define a tensor, analogous to the solder form defined in~\cite{BC2004,BR1990}, 
\[ \beta: T\Sigma \to Hom(f,f^{(1)}/f),\quad X\mapsto (\sigma \mapsto d_{X}\sigma \, mod\, f).\]
In accordance with~\cite[Theorem 4.3]{C2008} we have the following definition:

\begin{definition}
A map $f:\Sigma\to\mathcal{Z}$ is a Legendre map if it satisfies the contact condition, $f^{(1)}\le f^{\perp}$, and the immersion condition, $\ker\beta =\{0\}$. 
\end{definition}

\begin{remark}
\label{rem:legder}
The contact and immersion conditions together imply that $f^{(1)}=f^{\perp}$ (see~\cite{P1985}). 
\end{remark}

Note that $f^{\perp}/f$ is a rank $2$ subbundle of $\underline{\mathbb{R}}^{4,2}/f$ that inherits a positive definite metric from $\mathbb{R}^{4,2}$. 

\begin{definition}
\label{def:curvsph}
Let $p\in\Sigma$. Then a 1-dimensional subspace $s(p)\le f(p)$ is a curvature sphere of $f$ at $p$ if there exists a non-zero subspace $T_{s(p)}\le T_{p}\Sigma$ such that $\beta(T_{s(p)})s(p) = 0$. We call the maximal such $T_{s(p)}$ the curvature space of $s(p)$. 
\end{definition}

It was shown in~\cite{P1985} that at each point $p$ there is either one or two curvature spheres. We say that $p$ is an \textit{umbilic point of $f$} if there is exactly one curvature sphere $s(p)$ at $p$ and in that case $T_{s(p)}=T_{p}\Sigma$. 

Away from umbilic points we have that the curvature spheres form two rank 1 subbundles $s_{1},s_{2}\le f$ with respective curvature subbundles $T_{1}=\bigcup_{p\in \Sigma}T_{s_{1}(p)}$ and $T_{2}=\bigcup_{p\in \Sigma}T_{s_{2}(p)}$. We then have that $f=s_{1}\oplus s_{2}$ and $T\Sigma = T_{1}\oplus T_{2}$. A conformal structure $c$ is induced on $T\Sigma$ as the set of all indefinite metrics whose null lines are $T_{1}$ and $T_{2}$. This conformal structure induces a Hodge-star operator $\star$ that acts as $id$ on $T^{*}_{1}$ and $-id$ on $T^{*}_{2}$. 

Suppose that $f$ is umbilic-free. Then for each curvature subbundle $T_{i}$ we may define a rank 3 subbundle $f_{i}\le f^{\perp}$ as the set of sections of $f$ and derivatives of sections of $f$ along $T_{i}$. 
One can check that given any non-zero section $\sigma\in \Gamma f$ such that $\langle\sigma\rangle\cap s_{i} = \{0\}$ we have that
\[ f_{i} = f\oplus d\sigma(T_{i}).\]
Furthermore, 
\[ f^{\perp}/f = f_{1}/f\oplus_{\perp} f_{2}/f,\]
and each $f_{i}/f$ inherits a positive definite metric from that of $\mathbb{R}^{4,2}$. 

\begin{lemma}
\label{lem:innprods}
Let $X\in \Gamma T_{1}$ and $Y\in \Gamma T_{2}$ be nowhere zero. Then for any sections $\sigma,\tilde{\sigma}\in \Gamma f$,
$(d_{X}\sigma, d_{X}\tilde{\sigma})=0$ (or, $(d_{Y}\sigma, d_{Y}\tilde{\sigma})=0$) if and only if either $\sigma\in \Gamma s_{1}$ or $\tilde{\sigma}\in \Gamma s_{1}$ (respectively, $\sigma\in \Gamma s_{2}$ or $\tilde{\sigma}\in \Gamma s_{2}$).
\end{lemma}
\begin{proof}
Let $\sigma_{1}\in \Gamma s_{1}$ and $\sigma_{2}\in \Gamma s_{2}$ be lifts of the curvature sphere congruences. Then we may write $\sigma = \alpha \sigma_{1} + \beta \sigma_{2}$ and  $\tilde{\sigma} = \gamma \sigma_{1} + \delta \sigma_{2}$, for some smooth functions $\alpha, \beta, \gamma ,\delta$. Then 
\[ (d_{X}\sigma, d_{X}\tilde{\sigma}) = \beta\delta (d_{X}\sigma_{2}, d_{X}\sigma_{2}),\]
since $d_{X}\sigma_{1}\in \Gamma f$. Since $f_{2}/f$ inherits a positive definite metric from $\mathbb{R}^{4,2}$, we have that $(d_{X}\sigma_{2}, d_{X}\sigma_{2})$ is nowhere zero. Thus, $(d_{X}\sigma, d_{X}\tilde{\sigma})=0$ if and only if $\beta=0$ or $\delta =0$, i.e., $\sigma\in \Gamma s_{1}$ or $\tilde{\sigma}_{1}\in \Gamma s_{1}$. 
\end{proof}

\subsection{Dupin cyclides}
After spheres, Dupin cyclides are the next simplest object in Lie sphere geometry. One constructs them as follows: let $D$ be a 3-dimensional subspace of $\mathbb{R}^{4,2}$ which inherits an inner product of signature $(2,1)$ from $\mathbb{R}^{4,2}$. Then we have a splitting of $\mathbb{R}^{4,2}$ as 
\[ \mathbb{R}^{4,2} = D\oplus D^{\perp}.\]
One may then regularly parametrise the projective lightcones of $D$ and $D^{\perp}$ by maps $L: S^{1}\to \mathbb{P}(D)$ and $L^{\perp}:S^{1}\to \mathbb{P}(D^{\perp})$. Then one obtains a Legendre map 
\[ \mathcal{X}:S^{1}\times S^{1}\to \mathcal{Z}, \quad \mathcal{X}(u,v)= L(u)\oplus L^{\perp}(v).\]
The projection of $\mathcal{X}$ to any space form yields a parametrisation of a Dupin cyclide. Moreover, $L$ and $L^{\perp}$ are the curvature sphere congruences of $\mathcal{X}$. 

Dupin cyclides were originally defined by Dupin~\cite{D1822} as the envelope of a 1-parameter family of spheres tangent to three given spheres. In this way a Dupin cyclide is determined by these three spheres. This can be seen by letting $a,b,c\in\mathbb{P}(\mathcal{L})$ such that their span has signature $(2,1)$. Then letting $D=a\oplus b \oplus c$, one can construct a Dupin cyclide as above. Furthermore, $a$, $b$ and $c$ belong to one family of curvature spheres of the resulting Dupin cyclide and every curvature sphere in the other family is simultaneously tangent to $a$, $b$ and $c$. 

Now suppose that $f:\Sigma\to \mathcal{Z}$ is an umbilic-free Legendre map with curvature sphere congruences $s_{1}$ and $s_{2}$ and respective curvature subbundles $T_{1}$ and $T_{2}$. Let $\sigma_{1}\in\Gamma s_{1}$ and $\sigma_{2}\in\Gamma s_{2}$ be lifts of the curvature sphere congruences and let $X\in\Gamma T_{1}$ and $Y\in \Gamma T_{2}$. Then from Definition~\ref{def:curvsph} it follows immediately that 
\[ d_{X}\sigma_{1},d_{Y}\sigma_{2}\in\Gamma f.\]
Let 
\[ S_{1}:= \left\langle \sigma_{1},d_{Y}\sigma_{1},d_{Y}d_{Y}\sigma_{1}\right\rangle \quad \text{and}\quad 
S_{2}:= \left\langle \sigma_{2},d_{X}\sigma_{2},d_{X}d_{X}\sigma_{2}\right\rangle.\]
It was shown in~\cite{B1929} that $S_{1}$ and $S_{2}$ are orthogonal rank 3 subbundles of $\underline{\mathbb{R}}^{4,2}$ and the restriction of the metric on $\mathbb{R}^{4,2}$ to each $S_{i}$ has signature $(2,1)$. Furthermore, $S_{1}$ and $S_{2}$ do not depend on choices and we have the following orthogonal splitting
\[ \underline{\mathbb{R}}^{4,2} = S_{1}\oplus_{\perp} S_{2}\] 
of the trivial bundle. We refer to this splitting as the \textit{Lie cyclide splitting of $\underline{\mathbb{R}}^{4,2}$} because it can be identified with the Lie cyclides of $f$, i.e., a special congruence of Dupin cyclides making second order contact with $f$ (see~\cite[\S 86]{B1929}). 

This splitting now yields a splitting of the trivial connection $d$ on $\underline{\mathbb{R}}^{4,2}$:
\[ d = \mathcal{D} + \mathcal{N},\]
where $\mathcal{D}$ is the direct sum of the induced connections on $S_{1}$ and $S_{2}$ and 
\begin{equation}
\label{eqn:lcycN}
\mathcal{N} = d - \mathcal{D}\in \Omega^{1}((Hom(S_{1},S_{2})\oplus Hom(S_{2},S_{1}))\cap \mathfrak{o}(4,2)).
\end{equation}
Since $S_{1}$ and $S_{2}$ are orthogonal, we have that $\mathcal{D}$ is a metric connection on $\underline{\mathbb{R}}^{4,2}$ and $\mathcal{N}$ is a skew-symmetric endomorphism. Hence, $\mathcal{N}\in\Omega^{1}(S_{1}\wedge S_{2})$.

\subsection{Ribaucour transforms}
\label{subsec:ribaucour}

Suppose that $f,\hat{f}:\Sigma\to \mathcal{Z}$ are pointwise distinct Legendre immersions enveloping a common sphere congruence $s_{0}:=f\cap\hat{f}$. Assume that $f$ and $\hat{f}$ are umbilic-free with respective curvature sphere congruences $s_{1},s_{2}$ and $\hat{s}_{1},\hat{s}_{2}$, and let $T_{1},T_{2}\le T\Sigma$ and $\hat{T}_{1}, \hat{T}_{2}\le T\Sigma$ denote their respective rank 1 curvature subbundles. Classically two surfaces are Ribaucour transforms of each other if they are the envelopes of a sphere congruence such that the curvature directions of the surfaces are preserved. Interpreting this in the context of umbilic-free Legendre maps we have the following definition:

\begin{definition}
\label{def:rib}
Two umbilic-free Legendre maps $f,\hat{f}:\Sigma\to \mathcal{Z}$ are Ribaucour transforms of each other if $f$ and $\hat{f}$ envelope a common sphere congruence $s_{0}$ and $\hat{T}_{1} = T_{1}$ and $\hat{T}_{2} = T_{2}$. We then say that $f$ and $\hat{f}$ are a Ribaucour pair. 
\end{definition}

 In~\cite{BH2006}, the condition that two Legendre maps be Ribaucour transforms of each other was equated to the flatness of a certain normal bundle. In~\cite[Corollary 2.11, Remark 2.12]{P2017} this definition was shown to be equivalent to the following:
 
 \begin{lemma}
 \label{lem:ribpar}
 $f$ and $\hat{f}$ are Ribaucour transforms of each other if and only if for any sphere congruences $s\le f$ and $\hat{s}\le \hat{f}$ such that $s_{0}\cap s = s_{0}\cap \hat{s}=\{0\}$ one may choose lifts $\sigma\in \Gamma s$ and $\hat{\sigma}\in\Gamma \hat{s}$ such that $d\sigma,d\hat{\sigma}\in \Omega^{1}((s\oplus \hat{s})^{\perp})$. 
 \end{lemma}
 
 We now show that a Ribaucour pair is equipped with a special pair of Dupin cyclide congruences. 
 
\begin{proposition}
\label{prop:ribcyc}
Suppose that $s_{0}$ nowhere coincides with the curvature sphere congruences $s_{1},s_{2}$ and $\hat{s}_{1},\hat{s}_{2}$ of $f$ and $\hat{f}$, respectively. Then $f$ and $\hat{f}$ are Ribaucour transforms of each other if and only if
\begin{equation}
\label{eqn:ribcyc}
d\hat{\sigma}_{1}(\hat{T}_{2})\le s_{1}\oplus \hat{s}_{1}\oplus d\sigma_{1}(T_{2}) \quad \text{and}\quad
 d\hat{\sigma}_{2}(\hat{T}_{1})\le s_{2}\oplus \hat{s}_{2}\oplus d\sigma_{2}(T_{1}), 
\end{equation}
where $\sigma_{i}\in \Gamma s_{i}$ and $\hat{\sigma}_{i}\in \Gamma \hat{s}_{i}$. 
\end{proposition}
\begin{proof}
Suppose that $f$ and $\hat{f}$ are Ribaucour transforms of each other and thus $\hat{T}_{1}=T_{1}$ and $\hat{T}_{2}=T_{2}$. Then $d_{X}\sigma_{1}\in \Gamma f$ and $d_{X}\hat{\sigma}_{1}\in \Gamma\hat{f}$, for $X\in \Gamma T_{1}$, $\sigma_{1}\in \Gamma s_{1}$ and $\hat{\sigma}_{1}\in \Gamma \hat{s}_{1}$. One then deduces that 
\[ s_{1}\oplus \hat{s}_{1}\oplus d\sigma_{1}(T_{2}) = s_{1}\oplus \hat{s}_{1}\oplus d\hat{\sigma}_{1}(T_{2}) = \langle \sigma_{0}, d_{X}\sigma_{0}, d_{X}d_{X}\sigma_{0}\rangle^{\perp}, \]
for $\sigma_{0}\in \Gamma s_{0}$. Similarly, one has that 
\[ s_{2}\oplus \hat{s}_{2}\oplus d\sigma_{2}(T_{1}) = s_{2}\oplus \hat{s}_{2}\oplus d\hat{\sigma}_{2}(T_{1}) = \langle \sigma_{0}, d_{Y}\sigma_{0}, d_{Y}d_{Y}\sigma_{0}\rangle^{\perp}, \]
for $Y\in \Gamma T_{2}$. 

Conversely, suppose that~(\ref{eqn:ribcyc}) holds. Now for $\sigma_{0}\in\Gamma s_{0}$ and $X\in \Gamma T_{1}$, one has that $d_{X}\sigma_{0}\perp s_{1}\oplus \hat{s}_{1}\oplus d\sigma_{1}(T_{2})$. Thus, $0= (d_{X}\sigma_{0}, d_{\hat{Y}}\hat{\sigma}_{1})$, for any $\hat{\sigma}_{1}\in \Gamma \hat{s}_{1}$ and $\hat{Y}\in \Gamma \hat{T}_{2}$. Writing $X= \hat{X}+\mu \hat{Y}$, for some $\hat{X}\in \Gamma \hat{T}_{1}$ and smooth function $\mu$, one has that  
\[ 0 = (d_{\hat{X}}\sigma_{0}, d_{\hat{Y}}\hat{\sigma}_{1}) + \mu (d_{\hat{Y}}\sigma_{0}, d_{\hat{Y}}\hat{\sigma}_{1}) = \mu (d_{\hat{Y}}\sigma_{0}, d_{\hat{Y}}\hat{\sigma}_{1}),\]
since $d_{\hat{X}}\hat{\sigma}_{1}\in \Gamma \hat{f}$. By Lemma~\ref{lem:innprods}, $(d_{\hat{Y}}\sigma_{0}, d_{\hat{Y}}\hat{\sigma}_{1})\neq 0$, and thus $\mu=0$. Hence, $\hat{T}_{1} = T_{1}$. A similar argument shows that $\hat{T}_{2}=T_{2}$. Hence, $f$ and $\hat{f}$ are Ribaucour transforms of each other. 
\end{proof}

We now seek a geometric interpretation of the conditions in~(\ref{eqn:ribcyc}). Suppose that $(u,v)$ are local curvature line coordinates of $f$ about a point $p=(u_{0},v_{0})$ and consider the Dupin cyclide for which $s_{1}(u_{0},v_{0})$, $s_{1}(u_{0},v_{0}+\epsilon)$ and $\hat{s}_{1}(u_{0},v_{0})$ are contained in one family of curvature spheres, for sufficiently small $\epsilon\neq 0$. One obtains a Dupin cyclide $D_{1}(p)$ by taking the limit as $\epsilon$ tends to zero. In this way one obtains a smooth congruence $D_{1}$ of Dupin cyclides over $\Sigma$ and in fact this is represented as 
\[ D_{1} = s_{1}\oplus \hat{s}_{1}\oplus d\sigma_{1}(T_{2}).\]
On the other hand, suppose that $(\hat{u},\hat{v})$ are curvature line coordinates for $\hat{f}$ around $p=(\hat{u}_{0},\hat{v}_{0})$. One can consider the Dupin cyclide $\hat{D}_{1}(\hat{u}_{0},\hat{v}_{0})$ formed by taking the limit $s_{1}(\hat{u}_{0},\hat{v}_{0})$,  $\hat{s}_{1}(\hat{u}_{0},\hat{v}_{0})$ and $\hat{s}_{1}(\hat{u}_{0},\hat{v}_{0}+\epsilon)$ as $\epsilon$ tends to zero. We then obtain a second smooth congruence $\hat{D}_{1}$ of Dupin cyclides over $\Sigma$:
\[ \hat{D}_{1} = s_{1}\oplus \hat{s}_{1}\oplus d\hat{\sigma}_{1}(\hat{T}_{2}).\]
One then deduces that the first condition of~(\ref{eqn:ribcyc}) is equivalent to asking that the two Dupin cyclide congruences $D_{1}$ and $\hat{D}_{1}$ coincide. 
The second condition of~(\ref{eqn:ribcyc}) can be interpreted in terms of $s_{2}$ and $\hat{s}_{2}$ in an analogous way. 

Therefore, if $f$ and $\hat{f}$ are a Ribaucour pair of umbilic-free Legendre immersions, one obtains two special Dupin cyclide congruences enveloping $f$ and $\hat{f}$:

\begin{definition}
\label{def:ribcyc}
The Dupin cyclide congruences 
\[ D_{1} := s_{1}\oplus \hat{s}_{1}\oplus d\sigma_{1}(T_{2}) = s_{1}\oplus \hat{s}_{1}\oplus d\hat{\sigma}_{1}(T_{2})\]
and 
\[ D_{2} := s_{2}\oplus \hat{s}_{2}\oplus d\sigma_{2}(T_{1}) = s_{2}\oplus \hat{s}_{2}\oplus d\hat{\sigma}_{2}(T_{1})\]
will be called the Ribaucour cyclide congruences of the Ribaucour pair $f$ and $\hat{f}$. 
\end{definition}

As shown in the proof of Proposition~\ref{prop:ribcyc}, one has that 
\[ D_{1}^{\perp} = \langle \sigma_{0}, d_{X}\sigma_{0}, d_{X}d_{X}\sigma_{0}\rangle \quad \text{and}\quad 
D_{2}^{\perp} = \langle \sigma_{0}, d_{Y}\sigma_{0}, d_{Y}d_{Y}\sigma_{0}\rangle,\]
where $\sigma_{0}\in \Gamma s_{0}$, $X\in \Gamma T_{1}$ and $Y\in \Gamma T_{2}$. 

\section{Channel surfaces in Lie sphere geometry}
\label{sec:channel}
Channel surfaces have been a rich area of interest for many years. Examples of such surfaces include surfaces of revolution, tubular surfaces and Dupin cyclides. They are given simply by the following well-known definition:

\begin{definition}
A channel surface is the envelope of a 1-parameter family of spheres. 
\end{definition}

There are several characterisations of theses surfaces, for example, in Euclidean geometry they are the surfaces for which one family of curvature lines are circular or, equivalently, one of the principal curvatures is constant along the corresponding family of curvature lines, i.e., in terms of local curvature line coordinates $(u,v)$, $\kappa_{1,u}=0$ or $\kappa_{2,v}=0$. We shall study these surfaces in the context of Lie sphere geometry. 

A sphere curve can be realised as a map $s:I\to \mathbb{P}(\mathcal{L})$, where $I$ is a 1-dimensional manifold. We impose a regularity condition that ensures the existence of an envelope of $s$: the induced metric on $s^{(1)}/s$ is positive definite. We now seek a parametrisation of the envelope of this sphere curve. In order to do this we shall construct a Legendre map enveloping $s$. Firstly, let $V$ be a rank 3 subbundle of $I\times \mathbb{R}^{4,2}$ such that the induced metric on $V$ has signature $(2,1)$ and such that $s^{(1)}\le V$. Then $V^{\perp}$ is a rank 3 subbundle of $I\times \mathbb{R}^{4,2}$ and at each point $t\in I$ we may parametrise the projective light cone of $V^{\perp}$ by a map $\tilde{s}_{t}:S^{1}\to \mathbb{P}(\mathcal{L})$. Without loss of generality, we make the assumption that this is a regular parametrisation, i.e., the induced metric on $\tilde{s}_{t}^{(1)}/\tilde{s}_{t}$ is positive definite. We may extend this smoothly to all of $I$ to obtain a map 
\[ \tilde{s}:I\times S^{1}\to \mathbb{P}(\mathcal{L}).\]
We also extend the maps $s$ and $V$ trivially to maps on $I\times S^{1}$. 
\begin{lemma}
\label{lem:legchan}
$f:=s\oplus \tilde{s}$ is a Legendre map.
\end{lemma}
\begin{proof}
Since $s\le V$ and $\tilde{s}\le V^{\perp}$ we have that $f:=s\oplus \tilde{s}$ is a map from $I\times S^{1}$ into $\mathcal{Z}$. Furthermore, $s^{(1)}\le V\perp \tilde{s}$. Hence, $f$ satisfies the contact condition. The immersion condition follows from the regularity conditions of $s$ on $I$ and $\tilde{s}_{t}$ on $S^{1}$ for each $t\in I$.
\end{proof}

\begin{remark}
Suppose that $f=s\oplus \tilde{s}$ is a Legendre map arising from $V$. Suppose that $\overline{V}$ is another rank 3 subbundle of $I\times \mathbb{R}^{4,2}$ such that the induced metric on $\overline{V}$ has signature $(2,1)$ and such that $s^{(1)}\le \overline{V}$. Then 
\[ \bar{s}:=f\cap \overline{V}^{\perp}:I\times S^{1}\to \mathbb{P}(\mathcal{L})\]
is well-defined, and since $\overline{V}$ only depends on $I$, $\bar{s}(t,.)$ parametrises the projective lightcone of $\overline{V}_{t}^{\perp}$, for each $t\in I$. 
\end{remark}

Since $s$ only depends on $I$, one has that $s$ is a curvature sphere congruence of $f$ with curvature subbundle $T_{1}:= TS^{1}$. If one chooses $V = \langle \sigma, d_{Y}\sigma,d_{Y}d_{Y}\sigma\rangle$, where $\sigma\in \Gamma s$ and $Y\in \Gamma TI$, then the resulting $\tilde{s}$ will be the other curvature sphere congruence of $f$ with curvature subbundle $T_{2}:= TI$: this follows from the fact that for any lifts $\sigma\in \Gamma s$ and $\tilde{\sigma}\in \Gamma \tilde{s}$,
\[ 0= (d_{Y}d_{Y}\sigma, \tilde{\sigma})= - (d_{Y}\sigma, d_{Y}\tilde{\sigma}).\]
Then by Lemma~\ref{lem:innprods}, one has that either $\sigma\in \Gamma s_{2}$ or $\tilde{\sigma}\in \Gamma s_{2}$. However, our assumption that $s^{(1)}/s$ is positive definite means that $\tilde{\sigma}\in \Gamma s_{2}$. In this case, the splitting of the trivial bundle $\underline{\mathbb{R}}^{4,2}=V\oplus V^{\perp}$ is the Lie cyclide splitting of $f$. 

Conversely, suppose that $f:\Sigma\to \mathcal{Z}$ is an umbilic-free Legendre map such that one of the curvature sphere congruence $s_{i}$ is constant along the leaves of its curvature subbundle $T_{i}$, i.e., $d_{X}\sigma_{i}\in \Gamma s_{i}$ for $\sigma_{i}\in \Gamma s_{i}$ and $X\in \Gamma T_{i}$. Then $f$ envelopes a sphere congruence $s:=s_{i}$ that only depends on one parameter. Hence, $f$ parametrises a channel surface. 

\begin{proposition}
\label{prop:legchannel}
An umbilic-free Legendre map $f:\Sigma\to \mathcal{Z}$ parametrises a channel surface if and only if one of the curvature sphere congruences $s_{i}$ is constant along the leaves of its curvature subbundle $T_{i}$. 
\end{proposition}

In view of Proposition~\ref{prop:legchannel} we have the following definition:

\begin{definition}
$T_{i}$ is called a circular curvature direction of $f$ if $s_{i}$ is constant along the leaves of $T_{i}$. 
\end{definition}

Since the Lie cyclides of a Legendre map are given by
\[ S_{1}=\langle \sigma_{1}, d_{Y}\sigma_{1}, d_{Y}d_{Y}\sigma_{1}\rangle \quad \text{and}\quad S_{2}=\langle \sigma_{2}, d_{X}\sigma_{2}, d_{X}d_{X}\sigma_{2}\rangle ,\]
where $\sigma_{1}\in \Gamma s_{1}$, $\sigma_{2}\in \Gamma s_{2}$, $X\in \Gamma T_{1}$ and $Y\in \Gamma T_{2}$, one deduces the following corollary:

\begin{corollary}
An umbilic-free Legendre map $f:\Sigma\to \mathcal{Z}$ parametrises a channel surface if and only if the Lie cyclides of $f$ are constant along the leaves of one of the curvature subbundles $T_{i}$, i.e., $\mathcal{N}(T_{i}) = 0$. 
\end{corollary}

\subsection{Channel surfaces as $\Omega_{0}$-surfaces}
In~\cite{MN2006} a class of surfaces, called Lie applicable surfaces, are shown to be the only surfaces in Lie sphere geometry that admit second order deformations. It is shown that this class of surfaces naturally splits into two subclasses, $\Omega$-surfaces and $\Omega_{0}$-surfaces. This is the Lie sphere geometric analogue of $R$- and $R_{0}$-surfaces in projective geometry. Although $\Omega_{0}$-surfaces are objects of Lie sphere geometry, they were classically defined~\cite{D1911i,D1911ii,D1911iii} as those surfaces in space forms which satisfy 
\begin{equation}
\label{eqn:demom0}
\left( \frac{V}{U}\frac{\sqrt{E}}{\sqrt{G}}\frac{\kappa_{1,u}}{\kappa_{1}-\kappa_{2}}\right)_{v} =0 \quad \text{or}\quad  \left( \frac{U}{V}\frac{\sqrt{G}}{\sqrt{E}}\frac{\kappa_{2,v}}{\kappa_{1}-\kappa_{2}}\right)_{u}=0 ,
\end{equation}
for some functions $U$ of $u$ and $V$ of $v$, in terms of curvature line coordinates $(u,v)$, where $E$ and $G$ denote the usual coefficients of the first fundamental form and $\kappa_{1}$ and $\kappa_{2}$ denote the principal curvatures. Since channel surfaces in space forms are those that satisfy either $\kappa_{1,u}=0$ or $\kappa_{2,v}=0$, it is immediate that these surfaces are $\Omega_{0}$-surfaces and in fact any choice of functions $U$ and $V$ will satisfy~(\ref{eqn:demom0}). 

We shall use the following gauge-theoretic definition of $\Omega_{0}$-surfaces:

\begin{definition}[{\cite[Definition 3.1]{P2017}}]
A Legendre map $f:\Sigma\to \mathcal{Z}$ is an $\Omega_{0}$-surface if there exists a closed 1-form $\eta\in \Omega^{1}(f\wedge f^{\perp})$ such that $[\eta\wedge\eta]=0$ and 
\[ q(X,Y)= tr (\sigma\mapsto \eta(X)d_{Y}\sigma:f\to f)\]
is a non-zero degenerate quadratic differential. 
\end{definition}

In fact, given a closed 1-form $\eta\in \Omega^{1}(f\wedge f^{\perp})$, one obtains a family of such closed 1-forms called the \textit{gauge orbit} of $\eta$ by defining $\tilde{\eta}:= \eta - d\tau$, for any $\tau\in \Gamma (\wedge^{2}f)$. Furthermore, the quadratic differential is well defined on this gauge orbit, i.e., $\tilde{q}=q$, where $\tilde{q}$ denotes the quadratic differential of $\tilde{\eta}$. It was shown in~\cite{P2017} that for $\Omega_{0}$-surfaces there exists a special member of this gauge orbit called the \textit{middle potential} satisfying $\eta\in \Omega^{1}(s_{i}\wedge f^{\perp})$ for one of the curvature sphere congruences $s_{i}$, namely, 
\[ \eta = \sigma_{i}\wedge \star d\sigma_{i}\]
for some lift $\sigma_{i}\in \Gamma s_{i}$. In this case we say that $s_{i}$ is an isothermic curvature sphere congruence. 

Now suppose that a Legendre map $f$ parametrises a channel surface. Then by Proposition~\ref{prop:legchannel} one of the curvature spheres, say $s_{1}$, is constant along the leaves of $T_{1}$. We may choose a lift $\sigma_{1}$ of $s_{1}$ so that $d|_{T_{1}}\sigma_{1} =0$. Such a lift is determined up to multiplication by a function $\mu:\Sigma\to \mathbb{R}$ such that $d|_{T_{1}}\mu=0$. Now consider $d(\star d\sigma_{1})$. If we let $X\in \Gamma T_{1}$ and $Y\in \Gamma T_{2}$, then 
\begin{align*}
d(\star d\sigma_{1})(X,Y) &= d_{X}(\star d_{Y}\sigma_{1}) - d_{Y}(\star d_{X}\sigma_{1}) - \star d_{[X,Y]}\sigma_{1}\\
&= -d_{X}d_{Y}\sigma_{1} - d_{Y}d_{X}\sigma_{1} - d_{\star[X,Y]}\sigma_{1}\\
&= - 2d_{Y}d_{X}\sigma_{1} - d_{[X,Y] + \star[X,Y]}\sigma_{1}\\
&= 0,
\end{align*}
since $d|_{T_{1}}\sigma_{1} =0$ and $[X,Y] + \star[X,Y] \in \Gamma T_{1}$. Hence, $d(\star d\sigma_{1})=0$. This implies that the $f\wedge f^{\perp}$ valued 1-form 
\begin{equation}
\label{eqn:chanform}
\eta = \sigma_{1}\wedge \star d\sigma_{1}
\end{equation}
is closed. Since $\eta(T_{1})=0$, it follows trivially that $[\eta\wedge \eta]=0$. Furthermore, the quadratic differential 
\[ q(X,Y) = tr(\sigma\mapsto \eta(X)d_{Y}\sigma) = - (\star d_{X}\sigma_{1},d_{Y}\sigma_{1}),\]
is non-zero, taking values in $(T^{*}_{2})^{2}$. Hence, $f$ is an $\Omega_{0}$-surface.

In summary, we have seen in two different ways that:

\begin{proposition}
Channel surfaces are $\Omega_{0}$-surfaces. 
\end{proposition}

Given a function $\mu:\Sigma\to \mathbb{R}$ such that $d|_{T_{1}}\mu=0$, by defining $\tilde{\sigma}_{1}=\mu\sigma_{1}$ we have that 
\[ \tilde{\eta} = \tilde{\sigma}_{1}\wedge \star d\tilde{\sigma}_{1}\]
is a closed 1-form with values in $f\wedge f^{\perp}$. We then have that the quadratic differential $\tilde{q}$ satisfies $\tilde{q}=\mu^{2} q$. Therefore, since the quadratic differentials do not coincide we have that $\eta$ and $\tilde{\eta}$ do not belong to the same gauge orbit. 

It is well known that $\Omega_{0}$-surfaces, and more generally Lie applicable surfaces, constitute an integrable system, stemming from the presence of a 1-parameter family of flat connections:

\begin{lemma}[{\cite[Lemma 4.2.6]{C2012i}}]
Suppose that $\eta\in \Omega^{1}(f\wedge f^{\perp})$ is closed and $[\eta\wedge \eta]=0$. Then $\{d+t\eta\}_{t\in\mathbb{R}}$ is a 1-parameter family of flat connections. 
\end{lemma}

It was shown in~\cite{BHPR2017} that one may distinguish subclasses of surfaces amongst $\Omega$-surfaces by using polynomial conserved quantities of the aforementioned family of flat connections. Furthermore, it was shown that $\Omega_{0}$-surfaces possessing a constant conserved quantity project to tubular surfaces in certain space forms. We shall now investigate general polynomial conserved quantities of $\Omega_{0}$-surfaces. Firstly, let us recall the definition of a polynomial conserved quantity:

\begin{definition}
A non-zero polynomial $p = p(t)\in\Gamma \underline{\mathbb{R}}^{4,2}[t]$ is called a polynomial conserved quantity of $\{d+t\eta\}_{t\in\mathbb{R}}$ if $p(t)$ is a parallel section of $d+t\eta$ for all $t\in\mathbb{R}$.
\end{definition}

The following lemmata show that we may distinguish channel surfaces from general $\Omega_{0}$-surfaces by the presence of a polynomial conserved quantity:

\begin{lemma}
Channel surfaces admit an $\Omega_{0}$-structure with a linear conserved quantity. 
\end{lemma}
\begin{proof}
Let $\mathfrak{p}$ be a non-zero vector in $\mathbb{R}^{4,2}$. Let $\sigma_{1}\in \Gamma s_{1}$ be a lift of $s_{1}$ such that $d|_{T_{1}}\sigma_{1} = 0$. Then the lift $\tilde{\sigma}_{1}:=- \frac{1}{(\sigma_{1},\mathfrak{p})}\sigma_{1}$ satisfies $d|_{T_{1}}\tilde{\sigma}_{1}=0$ and $(\tilde{\sigma}_{1}, \mathfrak{p})=-1$. Now if we let $\tilde{\eta}:= \tilde{\sigma}_{1}\wedge \star d\tilde{\sigma}_{1}$ then $\tilde{\eta}$ is closed and $\mathfrak{p}+t\tilde{\sigma}_{1}$ is a linear conserved quantity of $d+t\tilde{\eta}$. 
\end{proof}

\begin{lemma}
An $\Omega_{0}$-surface with a polynomial conserved quantity is a channel surface.
\end{lemma}
\begin{proof}
Suppose that $f$ is an $\Omega_{0}$-surface with a closed 1-form $\eta = \sigma_{1}\wedge \star d\sigma_{1}$. Suppose further that $d+t\eta$ admits a polynomial conserved quantity $p(t)= p_{0}+ t p_{1} + ... + t^{d}p_{d}$ with $p_{d}\neq 0$. For a contradiction, let us assume that $f$ is not a channel surface. This implies that $s_{1}^{(1)}= f\oplus d\sigma_{1}(T_{2})$. 
Now, 
\[ 0=\eta p_{d}= (\sigma_{1}\wedge \star d\sigma_{1}) p_{d} = (\sigma_{1}, p_{d}) \star d\sigma_{1} - (\star d\sigma_{1}, p_{d}) \sigma_{1}.\]
Thus, $(\sigma_{1}, p_{d})=0$ and $(\star d\sigma_{1}, p_{d})=0$. Hence, $p_{d}\in \Gamma (s_{1}^{(1)})^{\perp}$. Thus, we may write 
\[ p_{d} = a\, \sigma_{1} + b\, \sigma_{2} + c\, d_{X}\sigma_{2},\] 
where $\sigma_{2}\in \Gamma s_{2}$ and $X\in \Gamma T_{1}$.
Since $dp_{d} = -\eta p_{d-1}$, one has that $dp_{d}\in \Omega^{1}(s_{1}^{(1)})$. This implies that 
\[ 0 = d_{X}p_{d}\, mod\, s_{1}^{(1)} = d_{X}c\, d_{X}\sigma_{2} + c\, d_{X}d_{X}\sigma_{2} + b\, d_{X}\sigma_{2}\, mod\, s_{1}^{(1)}.\]
Since $d_{X}\sigma_{2}$ and $d_{X}d_{X}\sigma_{2}$ are linearly independent, one has that $c=b=0$. Now, 
\[ a \, d\sigma_{1} \, mod\, s_{1} = - (\sigma_{1},p_{d-1}) \star d\sigma_{1}\, mod\, s_{1}.\]
This can only hold if $a = (\sigma_{1},p_{d-1}) =0$, which contradicts that $p_{d}\neq 0$. 

\end{proof}

\subsection{Calapso transforms of channel surfaces}

As previously mentioned, $\Omega_{0}$-surfaces have a rich transformation theory. One transformation that arises for these surfaces is the Calapso transformation. Suppose that $\eta\in\Omega^{1}(f\wedge f^{\perp})$ is closed and $[\eta\wedge \eta]=0$. Let $\{d+t\eta\}_{t\in \mathbb{R}}$ be the resulting 1-parameter family of flat connections. For each $t\in\mathbb{R}$, there exists a local orthogonal trivialising gauge transformation $T(t):\Sigma\to \textrm{O}(4,2)$, that is, 
\[ T(t)\cdot (d+t\eta) = d.\]

\begin{definition}
$f^{t}:=T(t)f$ is called a Calapso transform of $f$. 
\end{definition}

In~\cite{P2017} it was shown that $f^{t}$ is again a Lie applicable surface whose quadratic differential satisfies $q^{t}=q$. Furthermore, the curvature spheres of $f^{t}$ are given by $s_{1}^{t} = T(t)s_{1}$ and $s_{2}^{t}= T(t)s_{2}$ with respective curvature subbundles $T_{1}^{t}=T_{1}$ and $T_{2}^{t}=T_{2}$. 

Suppose that $f$ is a channel surface and, without loss of generality, suppose that $T_{1}$ is the circular curvature direction of $f$. Then the closed 1-form $\eta\in \Omega^{1}(f\wedge f^{\perp})$ constructed in~(\ref{eqn:chanform}) satisfies $\eta(T_{1})=0$. Now if $\sigma^{t}_{1}\in \Gamma s_{1}^{t}$, then $\sigma^{t}_{1} = T(t)\sigma_{1}$, for some $\sigma_{1}\in \Gamma s_{1}$. Hence, for $X\in \Gamma T_{1}$, 
\[ d_{X}\sigma^{t}_{1} = T(t)(d_{X}+ t\eta(X))\sigma_{1} = T(t)d_{X}\sigma_{1}\in \Gamma s^{t}_{1},\]
since $d_{X}\sigma_{1}\in \Gamma s_{1}$. Therefore, $s_{1}^{t}$ is constant along the leaves of $T_{1}$ and thus $f^{t}$ is a channel surface with circular direction $T_{1}$. 

\begin{theorem}
The Calapso transforms of channel surfaces are channel surfaces with the same circular curvature direction. 
\end{theorem}

\section{Ribaucour transforms of channel surfaces}
\label{sec:rib}

Blaschke proved the following result regarding Ribaucour transforms of channel surfaces:

\begin{theorem}[\cite{B1929}]
\label{thm:blaschke}
Ribaucour transforms of channel surfaces have one family of spherical curvature lines. 
\end{theorem}

This section is devoted to the case where the Ribaucour transform of a channel surface is again a channel surface. 

\begin{theorem}
\label{thm:main}
Suppose that $s, \hat{s}:I\to\mathbb{P}(\mathcal{L})$ are two regular sphere curves that never span a contact element, i.e., $s$ is nowhere orthogonal to $\hat{s}$. Then the envelopes of $s$ and $\hat{s}$ are Ribaucour transforms\footnote{That is, one can parametrise the envelopes of $s$ and $\hat{s}$ such that they are Ribaucour transforms of each other in the sense of Definition~\ref{def:rib}.} of each other, with corresponding circular curvature directions, if and only if $s^{(1)}\oplus \hat{s}=\hat{s}^{(1)}\oplus s$. 
\end{theorem}
\begin{proof}
Using the parametrisation of Section~\ref{sec:channel}, let $f,\hat{f}:I\times S^{1}\to \mathcal{Z}$ be Legendre maps enveloping $s$ and $\hat{s}$, respectively, such that $f$ and $\hat{f}$ are Ribaucour transforms of each other. Let $s_{0}:= f\cap \hat{f}$. Assuming that the circular curvature directions of $f$ and $\hat{f}$ correspond, one has that $s_{1}=s$ and $\hat{s}_{1}=\hat{s}$. It then follows by Proposition~\ref{prop:ribcyc} that 
\[ s^{(1)}\oplus \hat{s} = s_{1}\oplus \hat{s}_{1}\oplus d\sigma_{1}(T_{2})=  s_{1}\oplus \hat{s}_{1}\oplus d\hat{\sigma}_{1}(T_{2}) = \hat{s}^{(1)}\oplus s,\]
where $T_{2}= TI$. 

Conversely, suppose that $s^{(1)}\oplus \hat{s}=\hat{s}^{(1)}\oplus s=:V$. Since $V^{\perp}$ has signature $(2,1)$, we may at each point of $I$ parametrise the elements of the projective lightcone along $S^{1}$, i.e., we have $s_{0}:I\times S^{1}\to V^{\perp}$. Then let $f:=s_{0}\oplus s$ and $\hat{f}:=s_{0}\oplus \hat{s}$. $f$ and $\hat{f}$ are Legendre maps by Lemma~\ref{lem:legchan} and are Ribaucour transforms of each other because $T_{1}=\hat{T}_{1} = TS^{1}$ and $T_{2}=\hat{T}_{2}=TI$. 
\end{proof}

We can interpret the result of Theorem~\ref{thm:main} geometrically as follows. For $t\in I$ and sufficiently small non-zero $\epsilon$, $s(t)$, $s(t+\epsilon)$ and $\hat{s}(t)$ belong to one family of curvature spheres of a Dupin cyclide. By allowing $\epsilon$ to tend to zero, we obtain a unique Dupin cyclide at $t$. On the other hand by repeating the same process with $s(t)$, $\hat{s}(t)$ and $\hat{s}(t+\epsilon)$, we obtain another Dupin cyclide at $t$. By doing this for all $t\in I$, we obtain two $1$-parameter families of Dupin cyclides. The theorem states that the envelopes of $s$ and $\hat{s}$ are a Ribaucour pair if and only if these two families of Dupin cyclides coincide. 

We now have a result regarding when a general Ribaucour pair consists of channel surfaces. Recall from Definition~\ref{def:ribcyc} that associated to a Ribaucour pair of umbilic-free Legendre immersions are the Ribaucour cyclide congruences. In a straightforward manner, one deduces the following theorem:

\begin{theorem}\label{thm_Di}
Suppose that $f,\hat{f}:\Sigma\to \mathcal{Z}$ are a Ribaucour pair of umbilic-free Legendre immersions. Then $f$ and $\hat{f}$ are channel surfaces with corresponding circular curvature direction $T_{i}$ if and only if one of the Ribaucour cyclide congruences $D_{i}$ is constant along the leaves of $T_{i}$. 
\end{theorem}

In fact, if $D_{i}$ is constant along the leaves of $T_{i}$, then $D_{i}$ is exactly the 1-parameter family of Dupin cyclides arising from Theorem~\ref{thm:main}. 

\subsection{Darboux transforms}
We shall now recall the construction of Darboux transforms of $\Omega_{0}$-surfaces. 

Suppose that $f$ is an $\Omega_{0}$-surface with isothermic curvature sphere congruence $s_{1}$. Let $\eta\in \Omega^{1}(s_{1}\wedge f^{\perp})$ be the middle potential of $f$. We then have that $\{d+t\eta\}_{t\in \mathbb{R}}$ is a 1-parameter family of flat connections. The flatness of these connections implies that they admit many parallel sections. Suppose that $\hat{s}$ is a parallel subbundle of $d+m\eta$ for $m\in \mathbb{R}\backslash \{0\}$. Let $s_{0}:= f\cap \hat{s}^{\perp}$ and $\hat{f}:= s_{0}\oplus \hat{s}$. Then it was shown in~\cite{P2017} that $\hat{f}$ is a Legendre map and furthermore an $\Omega_{0}$-surface with isothermic curvature sphere congruence $\hat{s}$. We call $\hat{f}$ a \textit{Lie-Darboux transform of $f$ with parameter $m$}. 

\begin{theorem}
Any Lie-Darboux transformation of a channel surface is a channel surface with the same circular curvature direction.
\end{theorem}
\begin{proof}
Suppose that $f$ is a channel surface with circular direction $T_{1}$ and suppose that $\hat{s}$ is a parallel subbundle of $d+m\eta$, i.e., for some $\hat{\sigma}\in \Gamma \hat{s}$, $d\hat{\sigma}=-m\eta\hat{\sigma}$. Then since $f$ is a channel surface with circular direction $T_{1}$, one has that $\eta(T_{1})=0$. Thus, $d|_{T_{1}}\hat{\sigma}=0$. Hence, $\hat{s}$ is constant along the leaves of $T_{1}$ and thus, $\hat{f}$ is a channel surface with circular direction $T_{1}$. 
\end{proof}

\begin{theorem}\label{thm:darboux}
Given a Ribaucour pair of channel surfaces with corresponding circular curvature directions, we may choose their $\Omega_{0}$-structures so that this is a Lie-Darboux pair. 
\end{theorem}
\begin{proof}
Suppose that $f$ and $\hat{f}$ are a Ribaucour pair of channel surfaces with circular curvature direction $T_{1}$. By Lemma~\ref{lem:ribpar}, since $f$ and $\hat{f}$ are a Ribaucour pair, we may choose lifts $\sigma_{1}\in \Gamma s_{1}$ and $\hat{\sigma}_{1}\in \Gamma \hat{s}_{1}$ such that
\[ d\sigma_{1}, d\hat{\sigma}_{1}\in \Omega^{1}((s_{1}\oplus \hat{s}_{1})^{\perp}).\]
Without loss of generality, we may assume that $(\sigma_{1},\hat{\sigma}_{1})=-1$. Since $f$ and $\hat{f}$ are channel surfaces with circular direction $T_{1}$, we must also have that $d|_{T_{1}}\sigma_{1}=d|_{T_{1}}\hat{\sigma}_{1}=0$. Now let 
\[ \eta:= \sigma_{1}\wedge d\hat{\sigma}_{1}.\]
Then 
\[ d\eta = d\sigma_{1}\curlywedge d\hat{\sigma}_{1} = d|_{T_{1}}\sigma_{1}\curlywedge d|_{T_{2}}\hat{\sigma} +d|_{T_{2}}\sigma_{1}\curlywedge d|_{T_{1}}\hat{\sigma} =0.\]
Hence, $\eta$ defines an $\Omega_{0}$-structure on $f$. Furthermore, 
\[ d\hat{\sigma}_{1} + \eta\hat{\sigma}_{1} = d\hat{\sigma}_{1} + (\sigma_{1},\hat{\sigma}_{1})d\hat{\sigma}_{1}= 0.\]
Thus, $\hat{f}$ is a Lie-Darboux transform of $f$. 
\end{proof}

\section{Symmetry breaking}
\label{sec:sym}
In~\cite{BH2006} a definition of Ribaucour pairs of $k$-dimensional submanifolds in the conformal $n$-sphere is given. It is shown that for appropriately constructed Legendre lifts, two $k$-dimensional submanifolds are a Ribaucour pair if and only if there Legendre lifts form a Ribaucour pair. In this section, using Theorem~\ref{thm:main}, we quickly recover this result in the case of curves in the conformal 3-sphere. To do this, we break symmetry as explained in detail in \cite[Sec.\,2.2]{BHPR2017}. 

Let $\mathfrak{p} \in \mathbb{R}^{4,2}$ be a timelike vector. A curve in a conformal geometry $\langle \mathfrak{p} \rangle^\perp$ can be interpreted as a sphere curve $s: I \rightarrow \mathbb{P}(\mathcal{L})$, which takes values in $\langle \mathfrak{p} \rangle^\perp$. By the construction of Section~\ref{sec:channel}, one obtains a Legendre immersion parametrising this curve. Furthermore, $s$ is one of the curvature sphere congruences of this Legendre immersion. 

Conversely, suppose that $f$ is an umbilic-free Legendre map such that one of the curvature sphere congruences, say $s_{1}$, satisfies $s_{1}\perp \mathfrak{p}$. Thus, $s_{1}=f\cap \langle \mathfrak{p}\rangle^{\perp}$. Now $d_{X}\sigma_{1}\in \Gamma f$ for all $X\in \Gamma T_{1}$ and $\sigma_{1}\in \Gamma s_{1}$. On the other hand, $(d_{X}\sigma_{1},\mathfrak{p}) = d_{X}(\sigma_{1},\mathfrak{p})=0$, and thus $d_{X}\sigma_{1} \in \Gamma s_{1}$. Thus, $s_{1}$ is constant along the leaves of $T_{1}$ and projects to a curve in the conformal geometry of $\langle\mathfrak{p}\rangle^{\perp}$. We have thus arrived at the following proposition:
 
\begin{proposition}
An umbilic-free Legendre map parametrises a regular curve in the conformal geometry $\langle\mathfrak{p}\rangle^{\perp}$ if and only if one of the curvature sphere congruences $s_{i}$ satisfies $s_{i}\perp \mathfrak{p}$.
\end{proposition}

We now recall the definition of Ribaucour transforms of curves: 
\begin{definition}[\cite{BJMR2016,H2003}]
\label{def_rib_curves}
Two curves form a Ribaucour pair if they envelop a circle congruence.
\end{definition}

\begin{theorem}
\label{thm:curves}
Two non-intersecting regular curves are Ribaucour transforms of each other if and only if there exists a Ribaucour pair of Legendre maps parametrising these curves with corresponding circular curvature directions. 
\end{theorem}
\begin{proof}
Let $s,\hat{s}:I\to \mathbb{P}(\mathcal{L})$ be the corresponding curves in $\langle\mathfrak{p}\rangle^{\perp}$. By Theorem~\ref{thm:main}, there exists a Ribaucour pair of Legendre maps parametrising $s$ and $\hat{s}$ with corresponding circular curvature directions if and only if $s^{(1)}\oplus \hat{s} = \hat{s}^{(1)}\oplus s$. However, $s^{(1)}\oplus \hat{s}$ and $\hat{s}^{(1)}\oplus s$ both belong to the conformal geometry $\langle\mathfrak{p}\rangle^{\perp}$, and the condition $s^{(1)}\oplus \hat{s} = \hat{s}^{(1)}\oplus s$ is exactly the condition that $s$ and $\hat{s}$ envelope a circle congruence (see, \cite{BJMR2016}). In fact, the projective lightcone of $s^{(1)}\oplus \hat{s} = \hat{s}^{(1)}\oplus s$ yields exactly this circle congruence. 
\end{proof}

We may interpret Theorem~\ref{thm:curves} in Euclidean geometry as follows: two curves are Ribaucour transforms of each other if and only if tubes of the same radius over these curves are Ribaucour transforms of each other with corresponding circular curvature directions. We shall illustrate this with the following simple example. This example is generated by taking a Ribaucour transform of a straight line. By performing a parallel transformation, one obtains a Ribaucour transform of a cylinder. An explicit parametrisation of this Ribaucour transform is given in~\cite{T2002}. The tangent circles between the Ribaucour pair of curves become tori with the same radii as that of the tubular surfaces. These tori form the Ribaucour cyclide congruence that only depends on one parameter (see Theorem~\ref{thm_Di}). Furthermore, the black circles in Figure~\ref{fig:rib} illustrate how the circular curvature lines on the cylinder and its Ribaucour transform coincide with circular curvature lines on the enveloping tori. 

\begin{figure}[phbt]
\begin{center}
\begin{tabular}{cc}
\includegraphics[scale=0.3]{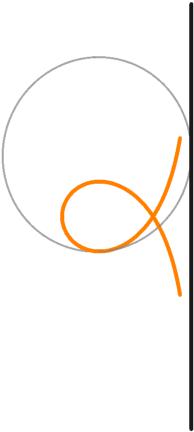} 
\quad\quad\quad\quad\quad&
\includegraphics[scale=0.3]{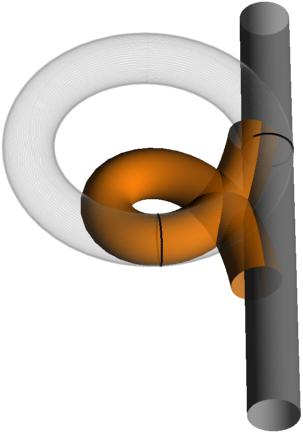} 
\end{tabular}
\caption
{Ribaucour transform of a straight line and, after parallel transformation, Ribaucour transform of a cylinder.}
\label{fig:rib}
\end{center}
\end{figure} 

\begin{remark}
Theorem \ref{thm:darboux} applied to the particular case of curves recovers a result given in \cite{BJMR2016}: for any Ribaucour pair of curves we can choose a polarization such that it becomes a Darboux pair. 
\end{remark}

In light of this section, we may reinterpret Theorem~\ref{thm:main} in the following way. Using isotropy projection (see for example~\cite{B1929,C2008}), one may view spheres as points in $\mathbb{R}^{3,1}$. Thus, sphere curves correspond to curves in $\mathbb{R}^{3,1}$. One may also view $\mathbb{R}^{4,2}$ as the conformal compactification of $\mathbb{R}^{3,1}$. The condition $s^{(1)}\oplus \hat{s} = \hat{s}^{(1)}\oplus s$ is then equivalent to the corresponding curves in $\mathbb{R}^{3,1}$ being Ribaucour transforms of each other.

\bibliographystyle{plain}
\bibliography{bibliography2017}

\end{document}